\documentclass[oneside, 12pt]{amsart}
\usepackage{amsmath}
\usepackage{amssymb}
\usepackage{dsfont}
\usepackage{color}
\usepackage[usenames,dvipsnames,x11names,svgnames]{xcolor}
\usepackage[colorlinks=true,linkcolor=NavyBlue,citecolor=DarkGreen]{hyperref}
\usepackage{amsthm}
\usepackage{amscd}
\usepackage{geometry}
\usepackage{mathrsfs}
\usepackage{mathtools}

\newtheorem{thm}{Theorem}

\newtheorem{prop}{Proposition}
\newtheorem{lem}[prop]{Lemma}
\newtheorem{cor}[prop]{Corollary}

\theoremstyle{definition}

\newtheorem*{rem}{Remark}

\newcommand{\mb}{\mathbb}
\newcommand{\mc}{\mathcal}

\newcommand{\ol}{\overline}
\newcommand{\leqs}{\leqslant}

\newcommand{\geqs}{\geqslant}

\newcommand{\NN}{\mathbb{N}}

\newcommand{\ind}{\mathds{1}}

\newcommand{\be}{\begin{equation*}}
\newcommand{\ee}{\end{equation*} }

\newcommand{\ben}{\begin{equation}}
\newcommand{\een}{\end{equation} }

\newcommand{\bs}{\begin{split}}
\newcommand{\es}{\end{split}}

\newcommand{\bmu}{\begin{multline*}}
\newcommand{\emu}{\end{multline*}}

\newcommand{\bmun}{\begin{multline}}
\newcommand{\emun}{\end{multline}}

\newcommand{\EE}{\mathbb{E}}
\newcommand{\PP}{\mathbb{P}}

\begin{document}
\title[Partial sums of random multiplicative functions]{Partial sums of random multiplicative functions and extreme values of a model for the Riemann zeta function}

\author{Marco Aymone}
\address{Departamento de Matem\'atica, Universidade Federal de Minas Gerais, Av. Ant\^onio Carlos, 6627, CEP 31270-901, Belo Horizonte, MG, Brazil.}
\email{aymone.marco@gmail.com}

\author{Winston Heap}
\address{Max Planck Institute for Mathematics, Vivatsgasse 7, 53111 Bonn.}
\email{winstonheap@gmail.com}

\author{Jing Zhao}
\address{Max Planck Institute for Mathematics, Vivatsgasse 7, 53111 Bonn.}
\email{jingzh95@gmail.com}

\begin{abstract}We consider partial sums of a weighted Steinhaus random multiplicative function and view this as a model for the Riemann zeta function. We give a description of the tails and high moments of this object. Using these we determine the likely maximum of $T\log T$ independently sampled copies of our sum and find that this is in agreement with a conjecture of Farmer--Gonek--Hughes on the maximum of the Riemann zeta function. We also consider the question of almost sure bounds. We determine upper bounds on the level of squareroot cancellation and lower bounds which suggest a degree of cancellation much greater than this which we speculate is in accordance with the influence of the Euler product. 

\end{abstract}
\maketitle

\section{Introduction}

In this paper we investigate a model for the Riemann zeta function provided by a sum of random multiplicative functions. To define these, let $(f(p))_{p}$ be a set of independent random variables uniformly distributed on the unit circle (Steinhaus variables) where $p$ runs over the set of primes and let $f(n)=\prod_{p^{v_p}|| n} f(p)^{v_p}$. Alternatively, one can take $(f(p))_{p}$ to be independent random $\pm 1$'s with equal probability (Rademacher variables), and let $f(n)$ be the multiplicative extension of these to the squarefree integers. 

The study of random multiplicative functions as a model for the usual deterministic multiplicative functions was initiated by Wintner \cite{W}. He considered the Rademacher case as a model for the M\"obius function and proved that the partial sums satisfy 
\begin{equation}
\label{as wint}
\sum_{n\leqs x}f(n)\ll x^{1/2+\epsilon}
\end{equation} 
almost surely, thus allowing him the assertion that ``Riemann's hypothesis is almost always true". We shall focus instead on the case of Steinhaus random multiplicative functions. In light of their orthogonality relations 
\[\mathbb{E}[f(m)\overline{f(n)}]={\bf 1}_{m=n}\] one can think of Steinhaus $f(n)$ as a model for $n^{it}$ with $t\in\mathbb{R}$. This point of view has been fruitfully used over the years with arguably the first instance being the pioneering work of Bohr \cite{bohr} (although the $f(p)$ appeared in a different guise there).  Given that 
\[
\zeta(\tfrac{1}{2}+it)\sim \sum_{n\leqs T}\frac{1}{n^{1/2+it}}
\]
for large $t\in [T,2T]$, the above reasoning suggests that for Steinhaus $f(n)$ the sum
\[
M_f(T)=\sum_{n\leqs T}\frac{f(n)}{\sqrt{n}}
\]
provides a good model for the zeta function. We investigate various aspects of this sum, starting with the value distribution of $|M_f(T)|$. 

In the case of the zeta function we have Selberg's famous central limit theorem which states that for $V=L\sqrt{\tfrac{1}{2}\log\log T}$ with $L\in\mathbb{R}$, fixed,
\[
\frac{1}{T}\mu\bigg( t\in[T,2T]:|\zeta(\tfrac{1}{2}+it)|\geqs e^V\bigg) \sim \frac{1}{\sqrt{\pi\log\log T }}\int_V^\infty e^{-x^2/\log\log T}dx
\]
as $T\to\infty$ where $\mu$ denotes Lebesgue measure. Regarding the uniformity of $V$, Selberg's original proof in fact allowed $V\ll(\log_2 T\log_3 T)^{1/2}$ which was recently improved to $V\ll (\log_2 T)^{3/5-\epsilon}$ by Radziwi\l\l\footnote{As stated, these results differ by those in the cited work by a factor of $\sqrt{\log_2 T}$ on account of our different normalisation.}  \,\cite{R}. It is expected that this asymptotic holds for all $V\ll \log_2 T$ and that beyond this range the distribution must change, if only slightly (see Conjecture 2 of \cite{R}). Jutila \cite{J} has given  Gaussian upper bounds in the range $0\leqs V\leqs \log_2 T$ whilst, under the assumption of the Riemann hypothesis, Soundararajan \cite{S} was able to extend similar bounds into the range $V\ll \log_2T \log_3T$. This allowed for near sharp bounds on the moments of the Riemann zeta function. 
For our sum $M_f(T)$ we prove the following. 

\begin{thm}\label{dist thm} Let $h(T)\to\infty$ arbitrarily slowly and suppose $(\log_2 T)^{1/2}\log_3 T\leqs V\leqs \log T/(\log\log T)^{h(T)}$. Then 
\begin{equation}
\label{large v range}
\mathbb{P}\big(|M_f(T)|\geqs e^V\big)=\exp\bigg(-(1+o(1))\frac{V^2}{\log(\tfrac{\log T}{V})}\bigg).
\end{equation}
If $V=L\sqrt{\tfrac{1}{2}\log\log T}$ with $L>0$ fixed then 
\begin{equation}
\label{small v range}
\mathbb{P}\big(|M_f(T)|\geqs e^V\big)\gg\int_L^\infty e^{-x^2/2}dx.
\end{equation}
\end{thm}   
\begin{rem}The range of $V$ in the lower bound \eqref{small v range} can be increased to $o(\log\log T)$ by applying large deviation theory in Lemma \ref{lemma gaussian lower bound} below (see Lemma 3.1 of \cite{aymone law of iterated logarithm for random dirichlet}). Since \eqref{small v range} is sufficient for our purposes, and there is only a small gap remaining in the range of $V$, we have left it as is.
\end{rem}

Thus, in contrast to the zeta function we are able to essentially understand the distribution in the range of larger $V$, whilst in the intermediate range the distribution is undetermined. The lower bound \eqref{small v range} suggests that it remains log-normal in this range, which would certainly be in analogy with the zeta function. Here, we remark that for the unweighted sum $\sum_{n\leqs T}f(n)$, Harper \cite{H high} has shown that there definitely \emph{is} a change in distribution around the intermediate range, going from something with tails of the order $e^{-2V}$ when $1\leqs V\leqs \sqrt{\log\log T}$, to something log-normal thereafter.  At any rate, we believe that in the larger range  $V\geqs \log\log T$, the estimate \eqref{large v range} should indeed reflect the true behaviour of the zeta function. Here we note the factor of $V^{-1}$ in the term $\log((\log T)/V)$ of \eqref{large v range} which becomes significant when $V\geqs (\log T)^{\theta}$ with $\theta>0$. 

As a quick corollary to these tail bounds we can derive ``likely" bounds for the maxima of independently sampled copies of $M_f(T)$. 

\begin{cor}
\label{sample max lem}
Let $f_1,\ldots,f_N$ be chosen independently. Then for $N=T\log T$ we have  
\begin{equation}
\label{P bound}
\mathbb{P}\bigg(\max_{1\leqs j\leqs N}\big|M_{f_j}(T)\big|\leqs \exp\big(\sqrt{(\tfrac{1}{2}+\epsilon)\log T\log\log T}\big)\bigg)=1-o_{T\to\infty}(1)
\end{equation}
for all $\epsilon>0$, whilst if $\epsilon<0$ the probability is $o(1)$. If $N=\log T$ then
\begin{equation}
\label{short interval bound}
\mathbb{P}\bigg(\max_{1\leqs j\leqs N}\big|M_{f_j}(T)\big|\leqs (1+\epsilon)\log T\bigg)=1-o_{T\to\infty}(1)
\end{equation}
for all $\epsilon>0$, whilst if $\epsilon<0$ the probability is $o(1)$.
\end{cor}
Since the zeta function at height $T$ oscillates on a scale of roughly $1/\log T$ (which can be seen either by considering its zeros or its approximation by a Dirichlet polynomial) one might expect that by sampling it at $T\log T$ independent points on the interval $[T,2T]$ one can pick up the maximum. From this point of view \eqref{P bound} represents a model for $\max_{t\in[T,2T]}|\zeta(\tfrac{1}{2}+it)|$ and is in agreement with a conjecture of Farmer--Gonek--Hughes \cite{FGH} which states that 
\[
\max_{t\in[T,2T]}|\zeta(\tfrac{1}{2}+it)|=\exp\Big((1+o(1))\sqrt{\tfrac{1}{2}\log T\log\log T}\Big).
\]
Similarly, \eqref{short interval bound} can be thought of as a short interval maximum $\max_{h\in[0,1]}|\zeta(\tfrac{1}{2}+it+ih)|$, $t\in [T,2T]$ and is in agreement with the leading order of a very precise conjecture of Fyodorov--Hiary--Keating \cite{FHK}. We remark that much work has gone into this latter conjecture, including a proof to leading order, independently by Arguin--Belius--Bourgade--Radziwi\l\l--Soundararajan \cite{ABBRS} and Najnudel \cite{N}, and an upper bound to second order by Harper \cite{Ha4}. 
  
We shall prove \eqref{large v range} of Theorem \ref{dist thm} by considering the moments of $|M(T)|$ whilst for \eqref{small v range}, which is just out reach with moment bounds, we rely on the methods of Harper \cite{H low}. The moments were initially considered by Conrey--Gamburd \cite{CG} who proved\footnote{The result of Conrey--Gamburd was proved for Dirichlet polynomials but by the Bohr correspondence their asymptotic formula applies to our sum of random multiplicative functions also.} that for fixed $k\in\mathbb{N}$, 
\[
\mathbb{E}[|M_f(T)|^{2k}]\sim c_k (\log T)^{k^2}  
\] 
where $c_k$ is an explicitly given constant. The case of real $k$ was considered by Bondarenko--Heap--Seip \cite{BHS} with refinements in the low moments case coming from Heap \cite{H} and then Gerspach \cite{G} who gave a fairly complete resolution of the problem by applying ideas from Harper's proof of Helson's conjecture \cite{H low}. As a result, we know that 
\begin{equation}
\label{pseudo bounds}
\mathbb{E}[|M_f(T)|^{2k}]\asymp_k (\log T)^{k^2}
\end{equation}
for all real, fixed $k> 0$. Concerning tail bounds, one often requires the moments in a uniform range of $k$. The case of large $k$ was considered in \cite{BH}, however the viable range of $k$ was somewhat lacking for the lower bounds. Here, we are able to fix this deficiency and prove the following. 

\begin{thm}
\label{moments thm}
For $10\leqs k\leqs C\log T/\log\log T$ we have  
\begin{equation}
\label{moment bounds}
\mathbb{E}[|M_f(T)|^{2k}]= (\log T)^{k^2} e^{-k^2\log k-k^2\log\log k+O(k^2)}.
\end{equation}
\end{thm}
We also give some partial results for $k$ in other ranges, including larger $k$ (see Proposition \ref{large moments prop}) and, by detailing Gerspach's \cite{G} proof for low moments, uniformly small $k$ (see Theorem \ref{Gerspach thm}). We remark that the proof of Theorem \ref{moments thm} is fairly elementary and does not require the probabilistic machinery of Harper \cite{H high} who proved bounds of the same quality for the unweighted sum $\sum_{n\leqs T}f(n)$. Our main tool is a hypercontractive inequality due to Weissler \cite{Weiss}.

Another motivation for this work was to investigate the problem of almost sure bounds. Due to its connection with partial sums of the M\"obius function, almost sure bounds for the sum $\sum_{n\leqs T}f(n)$ with $f(n)$ a Rademacher random multiplicative function have been extensively investigated. Improving the initial work of Wintner, in  an unpublished work Erd\"os showed that the almost sure bound in \eqref{as wint} can be improved to $\ll T^{1/2}(\log T)^A$. Hal\'asz \cite{Hal} then gave a significant improvement by proving the bound 
\begin{equation}
\label{hal bound}
\sum_{n\leqs T}f(n)\ll T^{1/2}\exp(c\sqrt{\log_2T\log_3 T}) \qquad \text{a.s.}
\end{equation}
Although the terms $f(n), f(n+1),\cdots $ are not necessarily independent, one might reasonably expect an almost sure bound on the level of the iterated logarithm, which would give $\sqrt{2T\log\log T}$. By carrying out a suggestion of Hal\'asz to remove the term $\log_3 T$ from the exponential in \eqref{hal bound}, Lau--Tenenbaum--Wu \cite{LTW} were in fact able to prove a result on this level by showing that 
\begin{equation}
\label{ltw bound}
\sum_{n\leqs T} f(n)\ll \sqrt{T}(\log\log T)^{2+\epsilon}\qquad \text{a.s.}
\end{equation}
Around the same time, Basquin \cite{bas} independently proved the same bound using a connection with sums over smooth numbers and an interesting observation interpreting these sums as martingales. 

Regarding omega theroems, the current best is due to Harper \cite{Ha0} who, improving on Hal\'asz \cite{Hal}, showed that almost surely
\begin{equation}
\label{harpers bound}
\sum_{n\leqs T} f(n)\neq O(\sqrt{T}(\log\log T)^{-5/2-\epsilon})
\end{equation}
for Rademacher $f(n)$. Likely, many of these results have similar counterparts for Steinhaus random multiplicative functions\footnote{Although perhaps with slightly smaller powers of the double logarithms since there is more chance of cancellation with Steinhaus variables.}. 

Turning to our case, as a first attempt one can apply the Rademacher--Menshov Theorem\footnote{Loosely, this states that if $\sum_{n=1}^{\infty}(\log n)^2\EE| X_n|^2<\infty$, then the series $\sum_{n=1}^\infty X_n$ converges almost surely. Thus $\sum_{n=1}^{\infty}\frac{f(n)}{\sqrt{n}(\log n)^{3/2+\epsilon}}$ converges almost surely, and hence, by partial summation the stated claim follows} to show that $M_f(T)\ll (\log T)^{3/2+\epsilon}$ almost surely.  Somewhat surprisingly, the machinery of Basquin \cite{bas} and Lau--Tenenbaum--Wu \cite{LTW} does not improve this by much since on applying a partial summation argument to \eqref{ltw bound} we get $M_f(T)\ll (\log T)(\log\log T)^{2+\epsilon}$ almost surely (at least, for Rademacher functions).
We are able to give a further improvement over this. 

\begin{thm}\label{theorem almost sure} For all $\epsilon>0$, the following
\begin{equation*}
M_f(T)\ll (\log T)^{1/2+\epsilon}
\end{equation*}
holds almost surely.
\end{thm}  

In terms of lower bounds, we prove the following. 
\begin{thm}
\label{teorema omega bound 1}
For all $L>0$ the following 
\begin{equation*}
\limsup_{T\to\infty}\frac{|M_f(T)|}{\exp((L+o(1))\sqrt{\log\log T})}\geq 1
\end{equation*}
holds almost surely.
\end{thm}

Thus, we have a considerable gap in our upper and lower bounds. The upper bound of Theorem \ref{theorem almost sure} is consistent with squareroot cancellation and represents the behaviour of a typical random sum. Indeed, one of the main inputs in the proof is a bound for the $(2+\epsilon)$th moment. If one could find a way to effectively input lower moments this could probably be improved, however we have not been able to do so. We note that from Chebyshev's inequality and bounds for low moments in \eqref{pseudo bounds} we get that 
\[
\mathbb{P}\Big(|M_f(T)|\leqs (\log T)^{\epsilon}\Big)=1-o(1)
\]
as $T\to\infty$ further suggesting that improvements of Theorem \ref{theorem almost sure} might be possible.

The lower bound of Theorem \ref{teorema omega bound 1} better displays the multiplicative nature of the problem. It suggests the sum is potentially being dictated by its Euler product since \[\prod_{p\leqs T}(1-f(p)p^{-1/2})^{-1}\approx \exp(\sum_{p\leqs T}f(p)p^{-1/2})\] and by the law of the iterated logarithm \cite{Kol} we have 
\[
\limsup_{T\to\infty} \frac{\Re\sum_{p\leqs T}f(p)p^{-1/2}}{\sqrt{\log_2 T\log_4T}}=1.
\]     

In any case, our proof of Theorem \ref{teorema omega bound 1} certainly relies on a connection with the Euler product. 
One of the main inputs is that the event $\mathcal{A}$ in which $M_f(T)\geq \exp((L+o(1))\sqrt{\log\log T})$ for infinitely many integers $T>0$ is a tail event, in the sense that any change to a finite set of values $(f(p))_{p\in\mathcal{S}}$, with $\mathcal{S}$ is a finite subset of primes, does not change the outcome. Since the values $(f(p))_p$ are independent, by the Kolmogorov zero--one law, $\mathcal{A}$ has probability either $0$ or $1$. By the Gaussian lower bound \eqref{small v range}, $\mathcal{A}$ must have positive probability, and hence, actually has probability $1$.

It is interesting to note that, again, the machinery of the bound \eqref{harpers bound} gives little more than $M_f(T)\neq O(1)$ almost surely, at least with a direct application.

\noindent \textbf{Acknowledgements.} The first author would like to thank Max Planck Institute for Mathematics for their warm hospitality during a visit in February 2020 (when this project started), and also the PPG/Mat - UFMG and CNPq (grant number 452689/2019-8) for financial support.


\section{Proof of Corollary \ref{sample max lem}} 

In this short section we deduce Corollary \ref{sample max lem} from Theorem \ref{dist thm}. Let us first deal with \eqref{P bound}.
Set $V=c\sqrt{\log T\log\log T}$ with $c>0$. By the independence of the trials, 
\begin{align*}
&
\mathbb{P}\big(\max_{1\leqs j\leqs T\log T}\big|M_{f_j}(T)\big|\leqs e^V\big)
= 
\mathbb{P}\big(\big|M_{f}(T)\big|\leqs e^V\big)^{T\log T}
\\ = & 
\Big(1-\mathbb{P}\big(\big|M_{f}(T)\big|> e^V\big)\Big)^{T\log T}
\\
= &
\exp\bigg(-T(\log T)\bigg[\mathbb{P}\Big(\big|M_{f}(T)\big|> e^V\big)+O\Big(\mathbb{P}\big(\big|M_{f}(T)\big|> e^V\big)^2\Big)\bigg] \bigg)
\end{align*}
By Theorem \ref{dist thm} we have 
\begin{align*}
\mathbb{P}\Big(\big|M_{f}(T)\big|> e^V\big)
= &
 \exp\Big(-(1+o(1))c^2\log T\log\log T/(\log(\log T/c\sqrt{\log T\log_2 T}))\Big)
 \\
 = &\exp(-(1+o(1))2c^2\log T)
\end{align*}
This is $o(1/T\log T)$ provided $c> 1/\sqrt{2}$ and hence our initial probability is $1-o(1)$. If $c<1/\sqrt{2}$ then our initial probability is $o(1)$. 
A similar proof gives \eqref{short interval bound}.


\section{Moment bounds} 

In this section we prove Theorem \ref{moments thm} and give some additional bounds for the moments in other ranges of $k$. We begin by proving Theorem \ref{moments thm}.

\subsection{Proof of Theorem \ref{moments thm}}

The implicit upper bound of Theorem \ref{moments thm} is from \cite{BH} and follows from Rankin's trick along with asymptotics for the tail sum $\sum_{p\geqs y}p^{-1-\sigma}$. As mentioned in the introduction, we only need to improve the range of $k$ in the lower bounds. We show that this, in fact, follows from the same essential ingredient which was a hyper-contractive inequality due to Weissler \cite{Weiss}. This can be stated as follows. For $\rho>0$ and a given random sum 
\[
F(T)=\sum_{n\leqs T} a_nf(n)
\]
with deterministic $a_n\in\mathbb{C}$, let
\[
F_{\rho}(T)=\sum_{n\leqs T} a_n f(n)\rho^{\Omega(n)}
\] 
where $\Omega(n)$ denotes the number of not-necessarily-distinct prime factors of $n$. Then the following appears in \cite[section 3]{bayart} (although in a slightly different form).

\begin{lem}[Weissler's inequality]\label{weissler lem} Let $0<p\leqs q$ and let $0\leqs \rho\leqs \sqrt{p/q}$. Then 
\[
\mathbb{E}[|F_{\rho}(T)|^q]^{1/q}\leqs \mathbb{E}[|F(T)|^p]^{1/p}.
\]
\end{lem}

This was originally proved for power series in one variable on the unit disk by Weissler \cite{Weiss}. Bayart \cite{bayart} then extended this to  multivariable power series using Minkowski's inequality. By the Bohr correspondence, these results apply to Dirichlet polynomials, or in our case, sums of random multiplicative functions.


\begin{lem}\label{lower moments lem} Let $k,T\geqs 10$. Then there exists a positive, absolute constant $A$ such that  
\[
{\mathbb{E}\big[\big|M_{f}(T)\big|^{2k}\big]}
\geqs 
(\log T)^{k^2}e^{-k^2\log k-k^2\log_2 k- Ak^2}.
\]
If $0<k\leqs 10$ then we may replace $e^{-k^2\log k-k^2\log_2 k- Ak^2}$ by some positive absolute constant $C$.
\end{lem}
\begin{proof}
By Weissler's inequality with $p=2k$, $q=2\lceil k\rceil$ and $\rho=\alpha_k:=\sqrt{k/\lceil k\rceil}$ we have for real $k> 0$,
\begin{equation}\label{weiss}
{\mathbb{E}\big[\big|M_{f}(T)\big|^{2k}\big]}\geqs {\mathbb{E}\bigg[\bigg|\sum_{n\leqs T}\frac{f(n)\alpha_k^{\Omega(n)}}{\sqrt{n}}\bigg|^{2\lceil k\rceil}\bigg]}^{k/\lceil k\rceil}.
\end{equation}
Let $K=\lceil k\rceil$ to ease notation. Then the expectation on the right hand side is given by 
\begin{align*}
{\mathbb{E}\bigg[\bigg|\sum_{n\leqs T}\frac{f(n)\alpha_k^{\Omega(n)}}{\sqrt{n}}\bigg|^{2K}\bigg]}
= &
\sum_{\substack{n_1\cdots n_K=n_{K+1}\cdots n_{2K}\\ n_j\leqs T}}\frac{\alpha_k^{\Omega(n_1)+\cdots \Omega(n_{2K})}}{(n_1\cdots n_{2K})^{1/2}}
\\
\geqs &
\sideset{}{^*}\sum_{\substack{n_1\cdots n_K=n_{K+1}\cdots n_{2K}\\ n_j\leqs T,\,\,n_j\in S(Y)}}\frac{\alpha_k^{\Omega(n_1)+\cdots \Omega(n_{2K})}}{(n_1\cdots n_{2K})^{1/2}} 
\end{align*}
where $\sum{}^*$ denotes the sum where the products $n_1\cdots n_k$ and $n_{k+1}\cdots n_{2k}$ are restricted to squarefree numbers and $S(Y)$ is the set of $Y$-smooth numbers with $Y\leqs T$. We proceed to remove the condition $n_j\leqs T$ in each summation variable. 

For a given $\delta>0$, the tail sum for $n_1$ takes the form 
\begin{align*}
\sideset{}{^*}\sum_{\substack{n_1\cdots n_K=n_{K+1}\cdots n_{2K}\\ n_1> T,\,\,n_j\leqs T,\,\,n_j\in S(Y)}}\frac{\alpha_k^{\Omega(n_1)+\cdots \Omega(n_{2K})}}{(n_1\cdots n_{2K})^{1/2}}
\leqs  &
\frac{1}{T^\delta}
\sideset{}{^*}\sum_{\substack{n_1\cdots n_K=n_{K+1}\cdots n_{2K}\\ n_j\in S(Y)}}\frac{\alpha_k^{\Omega(n_1)+\cdots \Omega(n_{2K})}}{n_1^{1/2-\delta}(n_2\cdots n_{2K})^{1/2}}
\\
= &
\frac{1}{T^\delta}
\prod_{p\leqs Y}\bigg(1+\frac{\alpha_k^2(p^\delta+K-1)\cdot K}{p}\bigg)
\\
= &
\frac{1}{T^\delta}
\prod_{p\leqs Y}\bigg(1+\frac{(p^\delta+K-1)\cdot k}{p}\bigg)
\end{align*}
where in the second line we have used that the condition $n_1\cdots n_K=n_{K+1}\cdots n_{2K}$ is multiplicative. By symmetry we acquire $2K$ such error terms. After removing the restrictions $n_j\leqs T$ in the main term we may write the resulting sum as an Euler product whose coefficient of $p^{-1}$ is $K^2\alpha_k^2=Kk$. Thereby, we obtain the lower bound
\[
\prod_{p\leqs Y}\bigg(1+\frac{Kk}{p}\bigg)-\frac{2K}{T^\delta}\prod_{p\leqs Y}\bigg(1+\frac{(p^\delta+K-1)\cdot k}{p}\bigg).
\]
In order to demonstrate the second term is little `oh' of the main term we consider the ratio
\[
\frac{2K}{T^\delta}\prod_{p\leqs Y}\frac{\big(1+(p^\delta+K-1)k/p\big)}{1+Kk/p}
\leqs 
\frac{2K}{T^\delta}\exp\Big(k\sum_{p\leqs Y}\frac{p^{\delta}-1}{p}\Big)
=
\frac{2K}{T^\delta}\exp\Big(O(k \delta  \log Y)\Big)
\]
provided $\delta\ll 1/\log Y$. If $k\geqs 10$ choose $\delta=1/\log Y$ and $Y=T^{1/(ck)}$ for some $c$. Then this ratio becomes $\exp(-ck+O(k))$ which is $\leqs1/2$ provided $c$ is large enough. If $0<k\leqs 10$ then we choose $\delta=1/\log Y$ and $Y=T^{1/c}$ for some $c$. In this case  the ratio is $\exp(-c+O(k))$ which again is $\leqs1/2$ provided $c$ is large enough. With these choices we acquire the lower bound
\begin{align*}
\frac{1}{2}\prod_{p\leqs Y}\bigg(1+\frac{Kk}{p}\bigg)
= &
\frac{1}{2}\prod_{p\leqs K^2}\bigg(1+\frac{Kk}{p}\bigg)\prod_{K^2<p\leqs Y}\bigg(1+\frac{Kk}{p}\bigg)
\\
= &
e^{O(K^2)} \exp\bigg(Kk\sum_{K^2<p\leqs Y}\frac{1}{p}+O\Big(k^4\sum_{p>K^2}p^{-2}\Big)\bigg)
\end{align*}
where we have used $\pi(K^2)\ll K^2/\log K$ in the first product. Using this again for the error term in the exponential, when $k\geqs 10$ we acquire the lower bound 
\[
e^{O(K^2)} \bigg(\frac{\log Y}{2\log K}\bigg)^{Kk}=(\log T)^{Kk}e^{-Kk\log k-Kk\log\log k+O(k^2)}
\]
since $Y=T^{1/ck}$ in this case. After raising this to the power $k/K$ the result follows in this range of $k$ by \eqref{weiss}. For $0<k\leqs 10$ the result follows similarly. 
\end{proof}


\subsection{Larger $k$}

\begin{prop}\label{large moments prop}
When $k\geqs c\log T/\log\log T$ we have 
\begin{equation}
\label{moment bounds large}
{\mathbb{E}\big[\big|M_{f}(T)\big|^{2k}\big]}
\leqs 
e^{Ck^2}\max(1, (\log T)^{k^2}e^{-k^2\log k}).
\end{equation}
for some positive absolute $C$.
\end{prop}
\begin{proof}
First suppose that $k$ is an integer. Then
\[
{\mathbb{E}\big[\big|M_{f}(T)\big|^{2k}\big]}=\sum_{\substack{n_1\cdots n_k=n_{k+1}\cdots n_{2k}\\n_j\leqs T}}\frac{1}{\sqrt{n_1\cdots n_{2k}}}=\sum_{n\leqs T^k}\frac{d_{k,T}(n)^2}{n}
\] 
where $d_{k,T}(n)=\sum_{n_1\cdots n_k=n,\, n_j\leqs T} 1$. Removing the divisor restriction $n_j\leqs T$ this is 
\[
\leqs \sum_{n\leqs T^k}\frac{d_k(n)^2}{n}\leqs  \sum_{n\leqs T^k}\frac{d_{k^2}(n)}{n}\leqs T^{k\sigma}\sum_{n\geqs 1}\frac{d_{k^2}(n)}{n^{1+\sigma}}=T^{k\sigma}\zeta(1+\sigma)^{k^2}
\]
for any $\sigma>0$ where in the second inequality we have used that $d_k(n)^2\leqs d_{k^2}(n)$ for $k\geqs 1$. This last inequality follows by comparison on prime powers and induction along with the formula 
$d_k(p^m)=\binom{k+m-1}{m}=\frac{1}{m!}(k+m-1)(k+m-2)\cdots k$. Choosing $\sigma=k/\log T$ and noting that $\zeta(1+\sigma)\ll\max(1/\sigma,1)$ the result follows for integer $k$. We can then interpolate to non-integral $k$ by using H\"older's inequality on noting that terms of the form $(\log T)^k$ are absorbed into $e^{O(k^2)}$.
\end{proof}


\subsection{Uniformly small $k$}

For upper bounds on uniformly small moments we make use of the recent progress of Gerspach \cite{G}. His result is stated for fixed $k$, however with a careful reading of the proof one can get uniform bounds. We will give the main details. Interestingly, it appears that there is a slight blow up of the constant as $k\to 0$. We do not know if this is an artefact of the proof or a result of some deeper change in the distribution around $V\approx \sqrt{\log\log T}$. 

\begin{thm}\cite{G}\label{Gerspach thm} Uniformly for $0<k\leqs 1$ we have 
\begin{align*}
\mathbb{E}[|M_f(T)|^{2k}] 
\leqs &
C (\log T)^{k^2} \min(\tfrac{1}{k^2}, \log\log T)\cdot\min(\tfrac{1}{k^2}, \log_3 T)
\\
&\qquad\qquad+\frac{2}{k}\min(\tfrac{1}{k},\log_3 T)
\end{align*}
for some absolute constant $C$.
\end{thm}   
\begin{proof}[Outline of modified proof]
One can check that the uniform version of Proposition 4 of \cite{G} is given by the inequality
\begin{multline*}
\mathbb{E}[|M_f(T)|^{2k}] 
\leqs 
\frac{A^k}{(\log T)^k}\sum_{0\leqs j\leqs J}\mathbb{E}\bigg[\bigg(\int_1^{T^{1-e^{-(j+1)}}}\bigg|\sum_{\substack{n>z\\P(n)\leqs T^{e^{-(j+1)}}}}\frac{f(n)}{\sqrt{n}}\bigg|^2\frac{dz}{z^{1-2j/\log T}}\bigg)^k\bigg]
\\
+B^k\sum_{0\leqs j\leqs J}e^{-ke^j}\mathbb{E}|F_j(1/2)|^{2k}+J\exp(-(1+o(1))k\sqrt{\log T})+1
\end{multline*}
where $A,B$ are positive absolute constants, $J=\lfloor\log_3 T\rfloor$ and 
\[F_j(s)=\prod_{p\leqs T^{e^{-j}}}(1-f(p)p^{-s})^{-1}.\] The manipulations of Proposition 5 which lead to the application of Parseval's theorem (e.g. \!see Theorem \ref{parseval thm} below) merely add an extra factor of $B^k$, and so, with a possibly different $A$, we find that the first term of the above is 
\[
\leqs 
\frac{A^k}{(\log T)^k}\sum_{0\leqs j\leqs J}\mathbb{E}\bigg[\bigg(\int_\mathbb{R} \frac{|F_j(\tfrac{1}{2}-\tfrac{2(j+1)}{\log T}+it)|^2}{|\tfrac{2(j+1)}{\log T}+it|^2}\bigg)^k\bigg].
\]
Now, uniformly for $0<k\leqs 1$ we have 
\[
\mathbb{E}|F_j(1/2)|^{2k}
=
\prod_{p\leqs T^{e^{-j}}}\sum_{m\geqs 0}\frac{d_k(p^m)^2}{p^m}
\leqs 
\prod_{p\leqs T^{e^{-j}}}\bigg(1+k^2\sum_{m\geqs 1}\frac{1}{p^m}\bigg)
\leqs 
C^{k^2}(\log T^{e^{-j}})^{k^2}
\]
where we have used $d_k(p^m)\leqs k$ which is valid for $m\geqs 1$ and $k$ in this range.  Therefore, on changing the constant $B$ from before, we arrive at the uniform bound
\begin{multline}
\label{pseudomoment}
\mathbb{E}[|M_f(T)|^{2k}] 
\leqs 
\frac{A^k}{(\log T)^k}\sum_{0\leqs j\leqs J}\mathbb{E}\bigg[\bigg(\int_\mathbb{R} \frac{|F_j(\tfrac{1}{2}-\tfrac{2(j+1)}{\log T}+it)|^2}{|\tfrac{2(j+1)}{\log T}+it|^2}\bigg)^k\bigg]
\\
+B^k(\log T)^{k^2}\sum_{0\leqs j\leqs J}e^{-ke^j}+J\exp(-(1+o(1))k\sqrt{\log T})+1.
\end{multline}
We now focus on the remaining expectation.

Following \cite{G}, we break the range of integration down into various sub-ranges. By symmetry in law, the expectation of the integral over $t<0$ is equal to that over $t>0$, so we focus on this latter range. We then break this down as 
\begin{align}
\label{breakdown}
\bigg[\int_{0<t\leqs \tfrac{j+1}{\log T}}
+\sum_{i=1}^{X}\int_{2^{i-1}\tfrac{j+1}{\log T}}^{2^{i}\tfrac{j+1}{\log T}}
+\sum_{i=1}^{Y}\int_{2^{i-1}\tfrac{e^j}{\log T}}^{2^{i}\tfrac{e^j}{\log T}}
+\sum_{i=1}^\infty\int_{2^{i-1}}^{2^i}\bigg]
\frac{|\mc{F}|^2}{|\tfrac{2(j+1)}{\log T}+it|^2}dt
\end{align}
where 
\[
X=\frac{\log(e^j/(j+1))}{\log 2},\qquad Y=\frac{\log (e^{-j}\log T)}{\log 2}
\]
and $\mc{F}=F_j({1}/{2}-{2(j+1)}/{\log T}+it)$ for short. Again by symmetry in law, the expectation of the first integral of \eqref{breakdown} is the same as the that of the first term of the first sum. Therefore, we concentrate on the ranges in these three sums. Combining  uniform versions of Propositions 10, 11 and 12 of \cite{G} we find that for $Z\geqs (j+1)/\log T$,
\begin{equation}
\begin{split}
\label{gerspach cases}
\mathbb{E}\bigg[\bigg(\int_Z^{2Z}\frac{|\mc{F}|^2}{|\tfrac{2(j+1)}{\log T}+it|^2}dt\bigg)^k\bigg]
\leqs
\frac{C}{Z^{2k}}
\cdot
\begin{cases}
e^{-jk^2}Z^{k}(\log T)^{k^2},\, & \tfrac{j+1}{\log T}\leqs Z \leqs \tfrac{e^j}{\log T} 
\\
e^{-jk}Z^{2k-k^2}(\log T)^{k}, &  \tfrac{e^j}{\log T} <Z\leqs 1
\\
e^{-jk}Z^{k}(\log T)^{k}, &  1 <Z.
\end{cases}
\end{split}
\end{equation}
These follow in the same way by applying Lemma 8 of \cite{G} which in fact holds for uniformly small exponents $b$ and $c$ there (see ``Euler product result 1" of \cite{H high}).

Applying \eqref{gerspach cases} we find that the expectation of the $k$th power of \eqref{breakdown} is bounded above by 
\begin{multline*}
C\bigg[e^{-jk^2}(\log T)^{k^2}\sum_{i=1}^{X}\bigg(2^{i-1}\frac{j+1}{\log T}\bigg)^{-k}
+e^{-jk}(\log T)^k\sum_{i=1}^{Y}\bigg(2^{i-1}\frac{e^j}{\log T}\bigg)^{-k^2}
\\+e^{-jk}(\log T)^k\sum_{i=0}^\infty 2^{-ik} \bigg]
\\
\leqs 
C\bigg[ Xe^{-jk^2} (\log T)^{k^2+k} + e^{-jk^2-jk} (\log T)^{k^2+k}\frac{1-2^{-Yk^2}}{1-2^{-k^2}}
+e^{-jk} (\log T)^{k}\frac{1}{1-2^{-k}}\bigg].
\end{multline*}
Since $X\leqs 2j$ and $Y\leqs 2\log\log T$, applying this in \eqref{pseudomoment} gives that $\mathbb{E}[|M_f(T)|^{2k}]$  is  
\begin{align*}
\leqs &CA^k\bigg[2(\log T)^{k^2}\sum_{0\leqs j\leqs J}je^{-jk^2}+(\log T)^{k^2}\min(\tfrac{1}{k^2}, Y)\sum_{0\leqs j\leqs J}e^{-2jk^2}
\\
&+\frac{2}{k}\sum_{0\leq j\leqs J}e^{-jk}\bigg]+B^k(\log T)^{k^2}\sum_{0\leqs j\leqs J}e^{-ke^j}+J\exp(-(1+o(1))k\sqrt{\log T})+1
\\
 \leqs &
C^\prime (\log T)^{k^2} \big(\min(\tfrac{1}{k^4},(\log_3 T)^2)+\min(\tfrac{1}{k^2}, \log\log T)\cdot\min(\tfrac{1}{k^2}, \log_3 T)\big)
\\
&\qquad\qquad+\frac{2}{k}\min(\tfrac{1}{k},\log_3 T)+\log_3 T\exp(-(1+o(1))k\sqrt{\log T})+1
\end{align*}
for some absolute constant $C^\prime$. Since the last two terms are of a lower order this is 
\begin{align*}
\leqs C^{\prime\prime} (\log T)^{k^2} \min(\tfrac{1}{k^2}, \log\log T)\cdot\min(\tfrac{1}{k^2}, \log_3 T)
+\frac{2}{k}\min(\tfrac{1}{k},\log_3 T)
\end{align*}
and so the result follows.
\end{proof}


\section{Tail bounds: Proof of Theorem \ref{dist thm}}

Theorem \ref{dist thm} consists of two statements. The first gives upper and lower bounds for the distribution in the range $(\log_2 T)^{1/2}\log_3T\leqs V\leqs \log T/(\log\log T)^{h(T)}$ whilst the second gives lower bounds  when $V=L\sqrt{\log_2 T}$ (small range). We further split the first of these into the ranges $(\log_2 T)^{1/2}\log_3T\ll V\leqs \log\log T$ (medium range) and $\log\log T\leqs V\ll \log T/\log\log T$ (large range). We will deal with these in order starting with the large range. 

\subsection{Large range $V$}

We begin with upper bounds since this is simpler. 
\begin{lem}\label{upper bound lem}
For  $\log\log T\leqs V\leqs C\log T/\log\log T$ we have
\[
\mathbb{P}\big(\big|M_{f}(T)\big|\geqs  e^V\big)\leqs \exp\bigg(-(1+o(1))\frac{V^2}{\log\big((\log T)/V\big)}\bigg).
\]
\end{lem}
\begin{proof}By Chebyshev's inequality and Theorem \ref{moments thm}  we have 
\[
\mathbb{P}\big(\big|M_{f}(T)\big|\geqs  e^V\big)\leqs \frac{\mathbb{E}\big[\big|M_{f}(T)\big|^{2k}\big]}{e^{2kV}}\leqs e^{-2kV}e^{-k^2\log k-k^2\log_2 k+O(k^2)}(\log T)^{k^2}
\]
provided $10\leqs k\leqs C\log T/\log\log T$. If $1\leqs k\leqs 10$ then the same bound holds with the factor $e^{-k^2\log k-k^2\log_2 k+O(k^2)}$ replaced by some absolute constant (by \eqref{pseudo bounds}).  Then for $10\log_2 T\leqs V\ll \log T/\log\log T$ we may take $k=V/\log((\log T)/V)$ in which case the right hand side becomes 
\begin{multline}
\label{upper}
\exp\bigg(-\frac{2V^2}{\log\big((\log T)/V\big)}+\frac{V^2}{\log^2\big((\log T)/V\big)}\Big(\log \log T
\\
-\log(V/\log((\log T)/V))
-\log_2(V/\log((\log T)/V)+O(1)\Big)\bigg)
\\
 \leqs
\exp\bigg(-\frac{V^2}{\log\big((\log T)/V\big)}+\frac{V^2}{\log^2\big((\log T)/V\big)}\log\log((\log T)/V)\bigg)
\end{multline}
which simplifies to the desired quantity. When $\log_2 T\leqs V\leqs 10\log_2 T$ the same choice of $k$ gives the result.
\end{proof}

The lower bounds is where we gain the slight restriction on the size of $V$ in the large range.

\begin{lem}\label{lower bounds lem}Suppose $\log\log T\leqs V\leqs C\log T/\log\log T$. If $V\leqs \log T/(\log\log T)^{h(T)}$ with $h(T)\to\infty$ arbitrarily slowly, then 
\begin{equation}
\label{lower 1}
\mathbb{P}\big(\big|M_{f}(T)\big|> e^V\big)\geqs \exp\Big(-(1+o(1))\frac{V^2}{\log((\log T)/V)}\Big).
\end{equation}
Otherwise, we have 
\begin{equation}
\label{lower 2}
\mathbb{P}\big(\big|M_{f}(T)\big|> e^V\big)\geqs \exp\Big(-(1+\epsilon)\frac{V^2}{\log((\log T)/V)}-\frac{10V^2\log_2 V}{\log^2((\log T)/V)}\Big).
\end{equation}
for any given fixed $\epsilon>0$.
\end{lem}
\begin{proof}
Let 
\[
\Phi_T(V)=\mathbb{P}\big(\big|M_{f}(T)\big|> e^V\big).
\]
Then 
\[
\mathbb{E}\big[\big|M_{f}(T)\big|^{2k}\big]=2k\int_0^\infty \Phi(\log u)u^{2k-1}du=2k\int_{-\infty}^{\infty}\Phi(u)e^{2ku}du.
\]
For a given $V$ we wish to show that there exists a $k=k_V$ and $\epsilon>0$ such that 
\[
\int_{V(1-\epsilon)}^{V(1+\epsilon)}\Phi(u)e^{2ku}du\sim \int_{-\infty}^{\infty}\Phi(u)e^{2ku}du.
\]
To motivate our choice of $k$ later we note that if indeed $\Phi(u)\approx e^{-u^2/\log(\log T/u)}$ then a quick check shows that  such a value of $k$ must occur at $k=V/\log(\log T/V)$.

Consider the upper tail. For this we have 
\begin{align*}
\int_{V(1+\epsilon)}^\infty \Phi(u)e^{2ku}du 
\leqs &
e^{-2k\delta V(1+\epsilon)}\int_{V(1+\epsilon)}^\infty \Phi(u)e^{2k(1+\delta)u}du
\\
\leqs &
e^{-2k\delta V(1+\epsilon)}\int_{-\infty}^\infty \Phi(u)e^{2k(1+\delta)u}du
\end{align*}
for any $\delta>0$. 
Again, we must consider separately the ranges $1\leqs k\leqs 10$ and $10\leqs k\leqs C\log T/\log\log T$ so that the double logarithms in Theorem \ref{moments thm} make sense. We consider the latter range since the former range can be dealt with similarly using the less complicated bounds $\mathbb{E}[|M_f(T)|^{2k}]\asymp (\log T)^{k^2}$. Continuing, by Theorem \ref{moments thm} the above is 
\begin{align*}
\leqs &
(\log T)^{k^2(1+\delta)^2}e^{-2k\delta V(1+\epsilon)-k^2(1+\delta)^2\log k-k^2(1+\delta)^2\log_2 k+ Ck^2}
\\
\leqs &
(\log T)^{k^2[(1+\delta)^2-1]}e^{-2k\delta V(1+\epsilon)-k^2[(1+\delta)^2-1]\log k-k^2[(1+\delta)^2-1]\log_2 k+ Dk^2} \int_{-\infty}^{\infty}\Phi(u)e^{2ku}du.
\end{align*}
The factor in front of the integral is 
\[
\exp\Big(-2k\delta V(1+\epsilon)+2\delta(1+\tfrac{\delta}{2})k^2\big(\log ((\log T)/k)-\log_2 k)+Dk^2\big)\Big).
\]
On choosing $k=V/\log(\log T/V)$ this becomes 
\begin{multline*}
\exp\Big(-2\delta(1+\epsilon)\frac{V^2}{\log((\log T)/V)}
\\
+2\delta(1+\tfrac{\delta}{2})\frac{V^2}{\log^2((\log T)/V)}\big(\log ((\log T)/V)-\log_2(V/\log((\log T)/V))\big)
\\
+D\frac{V^2}{\log^2((\log T)/V)}
\Big)
\end{multline*}
which simplifies to 
\begin{multline*}
\exp\Big(-2\delta(\epsilon-\tfrac{\delta}{2})\frac{V^2}{\log((\log T)/V)}
\\
-\frac{V^2}{\log^2((\log T)/V)}\Big[2\delta(1+\tfrac{\delta}{2})\log_2(\tfrac{V}{\log((\log T)/V)})-D\Big]
\Big).
\end{multline*}
Therefore, if we choose $\delta=\epsilon$ this has negative leading term in the exponential and hence is $o(1)$. 
Removing the double logarithm in the above we get the upper bound 
\[
\exp\Big(-\epsilon^2\frac{V^2}{\log((\log T)/V)}+D\frac{V^2}{\log^2((\log T)/V)}
\Big)
\]
which is still $o(1)$ provided $\epsilon\gg 1/\sqrt{\log((\log T)/V)}$. 

Now consider the lower tail. Applying a similar argument we have 
\begin{align*}
\int_{-\infty}^{V(1-\epsilon)} \Phi(u)e^{2ku}du 
\leqs &
e^{2k\delta V(1-\epsilon)}\int_{-\infty}^{V(1-\epsilon)} \Phi(u)e^{2k(1-\delta)u}du
\\
\leqs &
e^{2k\delta V(1-\epsilon)}\int_{-\infty}^\infty \Phi(u)e^{2k(1-\delta)u}du.
\end{align*}
By  \eqref{moment bounds} this is 
\begin{align*}
\leqs &
(\log T)^{k^2(1-\delta)^2}e^{2k\delta V(1-\epsilon)-k^2(1-\delta)^2\log k-k^2(1-\delta)^2\log_2 k+ Ck^2}
\\
\leqs &
(\log T)^{k^2[(1-\delta)^2-1]}e^{2k\delta V(1-\epsilon)-k^2[(1-\delta)^2-1]\log k-k^2[(1-\delta)^2-1]\log_2 k+ Dk^2} \int_{-\infty}^{\infty}\Phi(u)e^{2ku}du.
\end{align*}
The factor in front of the integral simplifies to
\[
\exp\Big(2k\delta V(1-\epsilon)-2\delta(1-\tfrac{\delta}{2})k^2\big(\log ((\log T)/k)-\log_2 k)+Dk^2\big)\Big).
\]
On setting $k=V/\log(\log T/V)$ this becomes 
\begin{multline*}
\exp\Big(2\delta(1-\epsilon)\frac{V^2}{\log((\log T)/V)}
\\
-2\delta(1-\tfrac{\delta}{2})\frac{V^2}{\log^2((\log T)/V)}\big(\log ((\log T)/V)-\log_2(V/\log((\log T)/V))\big)
\\
+D\frac{V^2}{\log^2((\log T)/V)}
\Big)
\end{multline*}
which simplifies to 
\begin{multline*}
\exp\Big(-2\delta(\epsilon-\tfrac{\delta}{2})\frac{V^2}{\log((\log T)/V)}
\\
+\frac{V^2}{\log^2((\log T)/V)}\Big[2\delta(1-\tfrac{\delta}{2})\log_2(\tfrac{V}{\log((\log T)/V)})+D\Big]
\Big).
\end{multline*}
Again, choosing $\delta=\epsilon$ this is $o(1)$, although this time with the proviso 
\begin{equation}
\label{epsilon cond}
\epsilon\gg \max(\sqrt{1/\log((\log T)/V)},(\log_2 V)/\log((\log T)/V)).
\end{equation}

We have therefore shown that for $k=V/\log((\log T)/V)$ and $\epsilon$ satisfying \eqref{epsilon cond}, 
\[
\int_{V(1-\epsilon)}^{V(1+\epsilon)}\Phi(u)e^{2ku}du\sim \int_{-\infty}^{\infty}\Phi(u)e^{2ku}du.
\]
Since $\Phi$ is a non-increasing function  we infer 
\[
2V\epsilon\Phi(V(1+\epsilon))e^{2kV(1-\epsilon)}
\leqs 
\int_{-\infty}^{\infty}\Phi(u)e^{2ku}du
\leqs
2V\epsilon\Phi(V(1-\epsilon))e^{2kV(1+\epsilon)}.
\]
For the right hand inequality, by Theorem \ref{moments thm} with the above choice of $k$, we have 
\begin{align*}
\Phi(V(1-\epsilon))
\geqs &
\frac{1}{4\epsilon kV}e^{-2kV(1+\epsilon)+k^2[\log_2 T-\log k-\log_2 k-A]}
\\
\geqs & \frac{\log((\log T)/V)}{V^2}\exp\Big(-(1+2\epsilon)\frac{V^2}{\log((\log T)/V)}
\\
&-\frac{V^2}{\log^2((\log T)/V)}[\log_2(V/\log((\log T/V)))+A ]\Big).
\end{align*}
This gives the second bound \eqref{lower 2} of the lemma. 

If $V\leqs \log T/(\log\log T)^{h(T)}$ with $h(T)\to\infty$ arbitrarily slowly set 
\[\epsilon=10\max(\sqrt{1/\log((\log T)/V)},(\log_2 V)/\log((\log T)/V))\]
and note this is $o(1)$. Then we get 
\begin{align*}
\Phi(V)
\geqs &\exp\Big(-\frac{V^2}{\log((\log T)/V)}
-\frac{100V^2}{\log^{2}((\log T)/V)}(\log_2 V+\sqrt{\log((\log T)/V)})\Big)
\end{align*}
and the first bound \eqref{lower 1} follows.
\end{proof}

\subsection{Medium range} In the medium range $(\log_2 T)^{1/2}\log _3 T\leqs V\leqs \log\log T$ we make use of bounds for low moments. Lemma \ref{lower moments lem} gives the lower bounds 
\begin{equation}
\label{uniform lower}
\mathbb{E}[|M_f(T)|^{2k}]\geqs C(\log T)^{k^2}
\end{equation}
uniformly in the range $0<k\leqs 1$ for some absolute constant $C>0$ whilst Theorem \ref{Gerspach thm} in the range $1/\sqrt{\log\log T}\leqs k\leqs 1$   gives the uniform bound
\begin{equation}
\label{uniform upper}
\mathbb{E}[|M_f(T)|^{2k}]\leqs  C(\log_2 T)(\log_3 T)(\log T)^{k^2} 
\end{equation}
for some (different) absolute constant $C>0$.

\begin{lem}[Medium range $V$] If $(\log_2 T)^{1/2}\log _3 T\leqs V\leqs \log\log T$  then
\[
\mathbb{P}\big(|M_f(T)|\geqs  e^V\big)= \exp(-(1+o(1))V^2/\log((\log T)/V)).
\]
\end{lem}
\begin{proof}
Given the range of $V$ it suffices to prove the bound 
\[
\mathbb{P}\big(|M_f(T)|\geqs  e^V\big)= \exp(-(1+o(1))V^2/\log\log T).
\]
For the upper bound, by Chebyshev's inequality and \eqref{uniform upper} we have 
\[
\mathbb{P}\big(|M_f(T)|\geqs e^V\big)\leqs C e^{-2kV}(\log_2 T)(\log_3 T)(\log T)^{k^2}
\]
for $1/\sqrt{\log\log T}\leqs k\leqs 1$. Choosing $k=V/\log\log T$ the right hand side is 
\[
\leqs 
\exp\Big(-\frac{V^2}{\log\log T}+2\log_3 T\Big)= \exp\Big(-(1+o(1))\frac{V^2}{\log\log T}\Big)
\]
given the range of $V$. 

For the lower bound we proceed similarly to Lemma \ref{lower bounds lem}. As before let 
\[
\Phi_T(V)=\mathbb{P}\big(\big|M_f(T)\big|> e^V\big)
\]
so that 
\begin{equation}
\label{law of}
\mathbb{E}\big[\big|M_f(T)\big|^{2k}\big]=2k\int_{-\infty}^{\infty}\Phi(u)e^{2ku}du.
\end{equation}
Let $0<k\leqs 1$ and $\epsilon>0$ to be chosen later.
Again we have
\begin{align*}
\int_{V(1+\epsilon)}^\infty \Phi(u)e^{2ku}du 
\leqs &
e^{-2k\delta V(1+\epsilon)}\int_{-\infty}^\infty \Phi(u)e^{2k(1+\delta)u}du
\end{align*}
for any $\delta>0$.
From \eqref{law of} and the bounds \eqref{uniform lower} and \eqref{uniform upper} this is 
\begin{align*}
\leqs &
C(2k)^{-1}(\log_2 T)(\log_3 T)(\log T)^{k^2(1+\delta)^2}e^{-2k\delta V(1+\epsilon)}
\\
\leqs &
C^\prime(\log_2 T)(\log_3 T)(\log T)^{k^2[(1+\delta)^2-1]}e^{-2k\delta V(1+\epsilon)} \int_{-\infty}^{\infty}\Phi(u)e^{2ku}du.
\end{align*}
The factor in front of the integral is 
\[
\leqs \exp\Big(2\delta(1+\tfrac{\delta}{2})k^2\log\log T-2k\delta V(1+\epsilon)+2\log_3 T\Big)
\]
which on choosing $k=V/\log\log Y$ and $\delta=\epsilon$ becomes 
\[
\exp\Big(-\epsilon^2\frac{V^2}{\log\log T}+2\log_3 T\Big)
\leqs 
\exp\Big(-\frac{\epsilon^2}{2}\frac{V^2}{\log\log T}\Big)
\]
provided $\epsilon\geqs 2/\sqrt{\log_3 T}$. A similar argument gives
\[
\int_{-\infty}^{V(1-\epsilon)} \Phi(u)e^{2ku}du 
\leqs 
\exp\Big(-\frac{\epsilon^2}{2}\frac{V^2}{\log\log T}\Big) \int_{-\infty}^{\infty}\Phi(u)e^{2ku}du.
\]

We have therefore shown that for $k=V/\log\log T$ and $\epsilon\geqs 2/\sqrt{\log_3 T}$ we have  
\[
\int_{V(1-\epsilon)}^{V(1+\epsilon)}\Phi(u)e^{2ku}du=(1+O(e^{-\epsilon^2V^2/2\log\log T})) \int_{-\infty}^{\infty}\Phi(u)e^{2ku}du.
\]
where the implicit constant may be taken to be 2. Since $\Phi$ is non-increasing and $V\geqs (\log_2 T)^{1/2}\log_3 T$ we infer 
\[
(1-O(e^{-\epsilon^2(\log_3 T)^2}))
\int_{-\infty}^{\infty}\Phi(u)e^{2ku}du
\leqs
2V\epsilon\Phi(V(1-\epsilon))e^{2kV(1+\epsilon)}.
\]
By \eqref{uniform lower} with the above choice $k=V/\log\log T$, we have 
\begin{align*}
\Phi(V(1-\epsilon))
\geqs &
C\frac{(1-O(e^{-\epsilon^2(\log_3 T)^2}))}{4\epsilon kV}e^{-2kV(1+\epsilon)}(\log T)^{k^2}
\\
= &
C(1-O(e^{-\epsilon^2(\log_3 T)^2}))
\frac{\log\log T}{4\epsilon V^2}\exp\Big(-(1+2\epsilon)\frac{V^2}{\log\log T}\Big).
\end{align*}
Choosing $\epsilon=2/\sqrt{\log_3 T}$ we get
\[
\Phi(V)\geqs \exp\Big(-(1+o(1))\frac{V^2}{\log\log T}\Big).
\]
\end{proof}

\subsection{Small range}\label{subsection small range} We now turn to proving the remaining lower bound \eqref{small v range} of Theorem \ref{dist thm} which states that for $V=(L+o(1))\sqrt{\log\log T}$ with $L>0$ fixed, 
\[
\mathbb{P}\big(|M(T)|\geqs e^V\big)\gg \int_{L}^\infty e^{-x^2/2}dx.
\]
 We make use of Harper's methods \cite{H low} following the proof of his Corollary 2 there.

We begin with the equivalent of Lemma 8 of \cite{H low}. Letting $\hat{\mathbb{E}}$ denote the conditional expectation with respect to the variables $(f(p))_{p\leqs \sqrt{T}}$ and $\hat{\mathbb{P}}$ the corresponding conditional probability, this states that if $\mathcal{A}$ denotes the event in which
\[
\bigg|\sum_{n\leqs T}\frac{f(n)}{\sqrt{n}}\bigg|\geqs \frac{1}{2}\hat{\mathbb{E}}\bigg|\sum_{n\leqs T, P(n)>\sqrt{T}}\frac{f(n)}{\sqrt{n}}\bigg|,
\] 
then $\hat{\mathbb{P}}(\mc{A})\gg 1$ for \emph{any} realisation of the $(f(p))_{p\leqs \sqrt{T}}$. We omit the proof of this since it follows more or less verbatim; the only difference being a factor of $1/\sqrt{n}$ in the sums. Then, since 
\[
\hat{\mathbb{E}}\bigg|\sum_{n\leqs T, P(n)>\sqrt{T}}\frac{f(n)}{\sqrt{n}}\bigg|
=
\hat{\mathbb{E}}\bigg|\sum_{\sqrt{T}<p\leqs T}\frac{f(p)}{\sqrt{p}}\sum_{n\leqs T/p}\frac{f(n)}{\sqrt{n}}\bigg|
\asymp
\bigg(\sum_{\sqrt{T}<p\leqs T}\frac{1}{p}\bigg|\sum_{n\leqs T/p}\frac{f(n)}{\sqrt{n}}\bigg|^2\bigg)^{1/2}
\]
by Khintchine's inequality, it suffices to prove the same lower bound for the probability
\[
\mathbb{P}\bigg(\sum_{\sqrt{T}<p\leqs T}\frac{1}{p}\bigg|\sum_{n\leqs T/p}\frac{f(n)}{\sqrt{n}}\bigg|^2\geqs e^{2V}\bigg).
\]  

Next we perform the smoothing step. With $X=\exp{(\sqrt{\log T})}$ write 
\begin{align*}
\sum_{\sqrt{T}<p\leqs T}\frac{1}{p}\bigg|\sum_{n\leqs T/p}\frac{f(n)}{\sqrt{n}}\bigg|^2
= &
\sum_{\sqrt{T}<p\leqs T}\frac{X}{p^2}\int_p^{p(1+1/X)}\bigg|\sum_{n\leqs T/p}\frac{f(n)}{\sqrt{n}}\bigg|^2dt
\\
\geqs &
\frac{1}{4}\sum_{\sqrt{T}<p\leqs T}\frac{X}{p^2}\int_p^{p(1+1/X)}\bigg|\sum_{n\leqs T/t}\frac{f(n)}{\sqrt{n}}\bigg|^2dt
\\
& -
\sum_{\sqrt{T}<p\leqs T}\frac{X}{p^2}\int_p^{p(1+1/X)}\bigg|\sum_{T/t< n\leqs T/p}\frac{f(n)}{\sqrt{n}}\bigg|^2dt
\end{align*}
where in the second line we have used the inequality $|a+b|^2\geqs \tfrac{1}{4}|a|^2-\min(|b|^2, \tfrac{1}{4}|a|^2)\geqs \tfrac{1}{4}|a|^2-|b|^2$. The expectation of the subtracted term here is 
\begin{align*}
\ll
\sum_{\sqrt{T}<p\leqs T}\frac{1}{p}\sum_{T/(p(1+1/X))< n\leqs T/p}\frac{1}{n}
= &
\sum_{\sqrt{T}<p\leqs T}\frac{1}{p}\Big(\log (1+1/X)+O(p/T)\Big)
\\
\ll &
\frac{\log\log T}{X} +\frac{1}{\log T}
\end{align*}
by the prime number theorem. Therefore, by Chebyshev's inequality the probability that this subtracted term is $\geqs e^{2V}/\sqrt{\log T}$, is  
$\ll \frac{1}{e^{2V}\sqrt{\log T}}$. Since this is much smaller than our target probability we can ignore this term. 

Returning to the first term we have 
\begin{align*}
\sum_{\sqrt{T}<p\leqs T}\frac{X}{p^2}\int_p^{p(1+1/X)}\bigg|\sum_{n\leqs T/t}\frac{f(n)}{\sqrt{n}}\bigg|^2dt
=
\int_{\sqrt{T}}^T\sum_{t/(1+1/X)<p\leqs t} \frac{X}{p^2}\bigg|\sum_{n\leqs T/t}\frac{f(n)}{\sqrt{n}}\bigg|^2dt
\end{align*}
and since by the prime number theorem
\[
\sum_{t/(1+1/X)<p\leqs t} \frac{X}{p^2}\geqs \frac{X}{\log t}\bigg(\frac{1}{t/(1+1/X)}-\frac{1}{t}\bigg)=\frac{1}{t\log t},
\]
this is 
\[
\geqs 
\int_{\sqrt{T}}^{ T}\bigg|\sum_{n\leqs T/t}\frac{f(n)}{\sqrt{n}}\bigg|^2\frac{dt}{t\log t}
\geqs 
\frac{1}{\log T}\int_{1}^{\sqrt{T}}\bigg|\sum_{n\leqs t}\frac{f(n)}{\sqrt{n}}\bigg|^2\frac{dt}{t}
\]
after letting $t\mapsto T/t$. We now note that we may add the condition $n\in S(T)$ in the sum with no change, and then after applying a small shift we find this is 
\[
\geqs \frac{1}{\log T}\int_{1}^{\sqrt{T}}\bigg|\sum_{\substack{n\leqs t\\n\in S(T)}}\frac{f(n)}{\sqrt{n}}\bigg|^2\frac{dt}{t^{1+\frac{4\log\log T}{\log T}}}.
\] 
Writing the integral as $\int_1^\infty -\int_{\sqrt{T}}^\infty$ we find that the expectation of the subtracted term is 
\[
\ll 
\frac{1}{\log T}\int_{\sqrt{T}}^\infty\bigg(\sum_{\substack{n\leqs t\\n\in S(T)}}\frac{1}{{n}}\bigg)\frac{dt}{t^{1+\frac{4\log\log T}{\log T}}}
\ll
\int_{\sqrt{T}}^\infty\frac{dt}{t^{1+\frac{4\log\log T}{\log T}}}
=
\frac{1}{4\log T\log\log T}.
\]
As before, this is seen to give a negligible contribution to the probability. Finally, we apply Parseval's Theorem for Dirichlet series. 

\begin{thm}[Parseval's Theorem, (5.26) of \cite{MV}]\label{parseval thm} For a given sequence of complex numbers $(a_n)_{n=1}^\infty$ consider the Dirichlet series $A(s)=\sum_{n=1}^\infty a_nn^{-s}$ and let $\sigma_c$ denote its abscissa of convergence. Then for any $\sigma>\max(0,\sigma_c)$ we have 
\[
\int_1^\infty \big|\sum_{n\leqs x}a_n\big|^2\frac{dx}{x^{1+2\sigma}}=\frac{1}{2\pi }\int_\mathbb{R}\frac{|A(\sigma+it)|^2}{|\sigma+it|^2}dt.
\]
\end{thm}

Applying this we find that
\[
\frac{1}{\log T}\int_{1}^{\infty}\bigg|\sum_{\substack{n\leqs t\\n\in S(T)}}\frac{f(n)}{\sqrt{n}}\bigg|^2\frac{dt}{t^{1+\frac{4\log\log T}{\log T}}}
=
\frac{1}{\log T}\int_{-\infty}^\infty \frac{|F(\tfrac{1}{2}+\tfrac{2\log\log T}{\log T}+it)|^2}{|\tfrac{2\log\log T}{\log T}+it|^2}dt
\]  
where $F(s)=\prod_{p\leqs T}(1-p^{-s})^{-1}$.

At this point we notice a difference to the case covered by Harper. The denominator of the integral on the right can get rather small around $t\approx 0$, which is not the case for the sum $\sum_{n\leqs T}f(n)$. To pick this up, we lower bound by the integral over the range $[-1/2\log T,1/2\log T]$. In this way we get the lower bound
\[
\geqs 
\frac{\log T}{4(\log_2 T)^2}\int_{-1/2\log T}^{1/2\log T} {|F(\tfrac{1}{2}+\tfrac{2\log\log T}{\log T}+it)|^2}dt
\]  
and have thus reduced the problem to the study of the probability 
\begin{equation}
\label{Prob 1}
\mathbb{P}\bigg(\log T\int_{-1/2\log T}^{1/2\log T}|F(\tfrac{1}{2}+\tfrac{2\log\log T}{\log T}+it)|^2dt\geqs e^{2V}(\log_2 T)^2\bigg).
\end{equation}

We now proceed similarly with Jensen's inequality although our ensuing analysis of the leading term is considerably simplified. We have 
\begin{align*}
&\log T\int_{-1/2\log T}^{1/2\log T}|F(\tfrac{1}{2}+\tfrac{2\log\log T}{\log T}+it)|^2dt
\\
= &
\log T\int_{-1/2\log T}^{1/2\log T}\exp(2\log |F(\tfrac{1}{2}+\tfrac{2\log\log T}{\log T}+it)|)dt
\\
\geqs  &
\exp\bigg(2\log T\int_{-1/2\log T}^{1/2\log T}\log |F(\tfrac{1}{2}+\tfrac{2\log\log T}{\log T}+it)|dt\bigg)
\\
= &
\exp\bigg(2\log T\int_{-1/2\log T}^{1/2\log T}\Re\sum_{p\leqs T}\bigg[\frac{f(p)p^{-it}}{p^{1/2+{2\log_2 T}/{\log T}}}
+\frac{f(p)^2p^{-2it}}{2p^{1+{4\log_2 T}/{\log T}}}\bigg]dt+O(1)\bigg)
\end{align*}
after applying the Euler product formula. 

Let us remove the second sum in the exponential. First note that since 
\[
\log T\int_{-1/2\log T}^{1/2\log T}p^{-2it}dt=1+O(\log p/\log T)
\]
and $\sum_{p\leqs x}\log p/p\ll \log x$ we can remove the term $p^{-2it}$ at a cost of $O(1)$. Now, 
\begin{equation}
\label{extra log}
\bigg|\sum_{p\leqs (\log T)^5}\frac{\Re f(p)^2}{p^{1+{4\log_2 T}/{\log T}}}\bigg|\leqs \sum_{p\leqs (\log T)^5}\frac{1}{p}=\log_3 T+O(1)
\end{equation}
and 
\[
\mathbb{E}\bigg|\sum_{(\log T)^5<p\leqs T}\frac{\Re f(p)^2}{p^{1+{4\log_2 T}/{\log T}}}\bigg|^2
\leqs 
\sum_{(\log T)^5<p\leqs T}\frac{1}{p^2}
\ll 
\frac{1}{(\log T)^6}
\]
by the Prime Number Theorem. Therefore,
\[
\mathbb{P}\bigg(\sum_{(\log T)^5<p\leqs T}\frac{\Re f(p)^2}{p^{1+{4\log_2 T}/{\log T}}}\geqs 1\bigg)\ll 1/(\log T)^6
\]
which, similarly to before, results in a negligible contribution compared to our target bound. Inputting these developments into \eqref{Prob 1} we have reduced the problem to lower bounding the probability
\[
\mathbb{P}\bigg(2\log T\int_{-1/2\log T}^{1/2\log T}\Re\sum_{p\leqs T}\frac{f(p)p^{-it}}{p^{1/2+{2\log_2 T}/{\log T}}}dt\geqs 2V+3\log_3 T\bigg)
\]
where the extra $\log_3 T$ term has come from \eqref{extra log}. We can now complete the proof with the following lemma. 

\begin{lem}\label{lemma gaussian lower bound} Let $L>0$ be a fixed constant and $V=(L+o(1))\sqrt{\frac{1}{2}\log\log T}$. Then
\begin{align*}
&
\mathbb{P}\bigg(  \sum_{p\leqs T}\log T\int_{-1/2\log T}^{1/2\log T}\frac{\Re\big(f(p)p^{-it}\big)}{p^{1/2+{2\log_2 T}/{\log T}}}dt\geqs V\bigg)\geq (1+o(1))\frac{1}{\sqrt{2\pi}}\int_L^\infty e^{-x^2/2}dx.
\end{align*}
\end{lem}
\begin{proof}
Set $f(p)=e^{i\theta_p}$, where $(\theta_p)_p$ are \textit{i.i.d.} with distribution uniform in the interval $[0,2\pi]$. Thus, $\Re(f(p)p^{-it})=\cos(\theta_p-t\log p)$, and hence
\begin{align*}
\log T\int_{-1/2\log T}^{1/2\log T}\Re\big(f(p)p^{-it}\big)dt&=\log T\int_{-1/2\log T}^{1/2\log T}\cos(\theta_p-t\log p)dt\\
&=\frac{\log T}{\log p}\left( \sin\left(\theta_p+\frac{\log p}{2\log T}\right)-\sin\left(\theta_p-\frac{\log p}{2\log T} \right)\right)\\
&=2\frac{\log T}{\log p}\sin\left(\frac{\log p}{2\log T}\right)\cos(\theta_p),
\end{align*}
where in the last equality we used that $\sin(a+b)=\sin(a)\cos(b)+\sin(b)\cos(a)$ and $\sin(a-b)=\sin(a)\cos(b)-\sin(b)\cos(a)$. Define
\begin{equation*}
\Sigma_T:= \sum_{p\leqs T}\log T\int_{-1/2\log T}^{1/2\log T}\frac{\Re\big(f(p)p^{-it}\big)}{p^{1/2+{2\log_2 T}/{\log T}}}dt.
\end{equation*}
Thus, we have shown that
\begin{equation*}
\Sigma_T= 2\sum_{p\leq T}\frac{1}{p^{1/2+2\log_2 T / \log T}}\frac{\log T}{\log p}\sin\left(\frac{\log p}{2\log T}\right)\cos(\theta_p).
\end{equation*}
Now notice that $(\cos(\theta_p))_p$ are \textit{i.i.d.} with $\EE \cos(\theta_p)=0$ and $\EE \cos(\theta_p)^2=1/2$. Then, since
$\sin(x)^2=\frac{1-\cos(2x)}{2}=x^2+O(x^4)$, we have
\begin{align*}
\EE \Sigma_T^2&= 2\sum_{p\leq T}\frac{1}{p^{1+4\log_2 T / \log T}}\frac{(\log T)^2}{(\log p)^2}\left(\frac{(\log p)^2}{4(\log T)^2}+O\left(\frac{(\log p)^4}{(\log T)^4}\right)\right)\\
&=\frac{1}{2}\sum_{p\leq T}\frac{1}{p^{1+4\log_2 T / \log T}} +O(1)
\end{align*}
by Chebyshev's estimate $\sum_{p\leqs x}\log p/p\ll \log x$. Splitting the sum at $p=T^{1/\log\log T}$ we find that 
\begin{align*}
\sum_{p\leqs T^{1/\log\log T}}\frac{1}{p^{1+4\log_2T/\log T}}
= &
\sum_{p\leqs T^{1/\log\log T}}\frac{1}{p}\bigg(1+O\Big(\frac{\log p \log\log T}{\log T}\Big)\bigg)
\\
= &
\sum_{p\leqs T^{1/\log\log T}}\frac{1}{p}+O(1)
\\
= &
\log\log T-\log_3 T+O(1)
\end{align*}
by Chebyshev's estimate again and then Mertens' Theorem. Since the tail sum is 
\[
\ll\sum_{T^{1/\log\log T}<p\leqs T}\frac{1}{p}\ll \log_3 T
\] 
we find that $\EE \Sigma_T^2=(1/2)\log\log T+O(\log_3 T)$. 

Now observe that each factor in $\Sigma_T$ is bounded by $1$, due to the fact that $|\sin(x)|\leq x$. Hence, by the Central Limit Theorem for triangular arrays (see, for instance, \cite{shiryaev}, pg. 334, Theorem 2), we have that
\begin{align*}
\frac{\Sigma_T}{\sqrt{\frac{1}{2}\log\log T}}\to_d\mathcal{N}(0,1),
\end{align*}
as $T\to\infty$ (where $\to_d$ means convergence in distribution).
\end{proof}

\section{Almost sure bounds}
In this section we are going to prove Theorem \ref{theorem almost sure} using the Borel--Cantelli Lemma as our main tool. As is typical, this consists of two main steps: a ``sparsification" step where the set of points $T$ is discretised to some sparser subset, and then a step where we bound the resultant probabilities (typically via Chebyshev's inequality and moment bounds). The sparsification step of Lau--Tenenbaum--Wu \cite[Lemma 2.3]{LTW} loses a factor of a logarithm which is crucial for us. Instead, we make use of a theorem of \cite{billingsleyconvergenceofprobabilitymeasures} (Lemma \ref{lemma billingsley} below) which involves a finer analysis. Our first two lemmas below provide the necessary moment bounds.   
\begin{lem}\label{lemma jing moment} Let $(a(n))_{n\in \NN}$ be a sequence of complex numbers such that $a(n)\neq0$ only for a finite number of $n$. Then, for any non-negative integer $\ell$, we have that
\begin{equation*}
\EE\bigg{|}\sum_{n\geqslant 1}\frac{a(n)f(n)}{\sqrt{n}} \bigg{|}^{2\ell}\leq \left(\sum_{n\geqslant 1}\frac{|a(n)|^2\tau_\ell(n)}{n} \right)^\ell,
\end{equation*}
where $\tau_\ell(n)=\sum_{n_1\cdot...\cdot n_\ell=n}1$.
\end{lem}
\begin{proof}We have
\begin{align*}
 \mb E \Big( \Big|\sum_{n\geqs 1} a(n)\frac{f(n)}{\sqrt{n}}\Big|^{2\ell} \Big) &=  \mb E \Big( \Big|\sum_{n_1, \dots, n_\ell \geqs 1} a(n_1)\cdots a(n_\ell)\frac{f(n_1)\cdots f(n_\ell)}{\sqrt{n_1\cdots n_\ell}}\Big|^{2} \Big)\\
 & = \sum_{n_1\cdots n_\ell = m_1\cdots m_\ell} \frac{a(n_1)\cdots a(n_\ell) \ol{a(m_1)}\cdots \ol{a(m_\ell)} }{\sqrt{n_1\cdots n_\ell \cdot m_1\cdots m_\ell}} \\
 & = \sum_{n\geqslant 1} |\sum_{n_1\cdots n_\ell =n}  a(n_1)\cdots a(n_\ell)   |^2/n.
\end{align*}
By the Cauchy--Schwarz inequality this is bounded above by 
 \begin{align*}
 \sum_{n\geqs 1} \frac{ \tau_\ell (n)}{n} \sum_{n_1\cdots n_\ell =n}  | a(n_1)\cdots a(n_\ell)   |^2 
 & \leqslant \sum_{n\geqs 1} \frac{1}{n} \sum_{n_1\cdots n_\ell =n}  | a(n_1)\cdots a(n_\ell)   |^2\tau_\ell(n_1)\cdots\tau_\ell(n_\ell) \\
 & = \Big( \sum_{n\geqs 1} \frac{|a(n)|^2\tau_\ell(n)}{n} \Big)^\ell
 \end{align*}
 where we have used $\tau_\ell(n_1\cdots n_\ell)\leqs \tau_\ell(n_1)\cdots \tau_\ell(n_\ell)$.
\end{proof}

\begin{lem}\label{uniform divisor lem}Uniformly for $1\leqs u< v$ we have 
\[
\sum_{u<n\leqs v}\frac{\tau(n)}{n} \leqs C (\log v)^{4/3} (\log(v/u))^{2/3}
\]
 for some positive absolute constant $C$. 
\end{lem}
\begin{proof}
We use classical improvements to Dirichlet's bound for the error term in the divisor problem. From Theorem 12.4 of \cite{T}, for example, we have 
\[
D(x):=\sum_{n\leqs x} \tau(n)=x\log x+(2\gamma-1)x+O(x^{1/3}).
\]
Therefore, by partial summation 
\begin{align*}
\sum_{u<n\leqs v} \frac{\tau(n)}{n}
= &
\int_u^v \frac{1}{x}\cdot dD(x)
= \frac{D(x)}{x}\bigg|_u^v+ \int_u^v \frac{D(x)}{x^2} dx
\\
= &
\log(v/u)+O(u^{-2/3})+\int_u^v \Big[\frac{\log x}{x}+\frac{2\gamma-1}{x}+O(x^{-5/3})\Big]dx
\\ 
= &
\frac{1}{2}(\log^2 v-\log^2 u) +2\gamma \log(v/u)+O(u^{-2/3})
\end{align*}
uniformly for $1\leqs u<v$. The first term of the above is $\leqs \log v \log(v/u)$ whilst the third is
$
u^{-2/3}\ll (\log \tfrac{u+1}{u})^{2/3}\leqs (\log \tfrac{v}{u})^{2/3}. 
$
Accordingly, 
\begin{align*}
\sum_{u<n\leqs v} \frac{\tau(n)}{n}
\ll &
 \log v \log(v/u) +\log(v/u)+(\log(v/u))^{2/3}
 \\
 \ll & (\log v)^{4/3}(\log(v/u))^{2/3}.
\end{align*}
\end{proof}

\begin{lem}[Theorem 10.2 of \cite{billingsleyconvergenceofprobabilitymeasures}, pg. 107]\label{lemma billingsley} Let $S_N=X_1+...+X_N$. Suppose that $\alpha>1/2$ and that $\beta\geq 0$. Let $u_1,...,u_N$ be real numbers such that for all $\lambda>0$
\begin{equation*}
\PP(|S_j-S_i|\geq \lambda)\leq\frac{1}{\lambda^{4\beta}}\left(\sum_{i<n\leq j}u_n\right)^{2\alpha},\quad 0\leq i\leq j \leq N.
\end{equation*}
Then
\begin{equation*}
\PP\left(\max_{k\leq N}|S_k|\geq \lambda \right)\ll_{\alpha,\beta}\frac{1}{\lambda^{4\beta}}\left(\sum_{n\leq N}u_n\right)^{2\alpha}.
\end{equation*}
\end{lem}

\begin{proof}[Proof of Theorem \ref{theorem almost sure}] Let $T_j=\exp(j^4)$.
We have, by Lemma \ref{lemma jing moment}, for any $T_{j-1} <u\leq v\leq T_j$ that
\begin{equation*}
\EE|M_f(v)-M_f(u)|^4\leq \left( \sum_{u<n\leq v } \frac{\tau(n)}{n}\right)^2.
\end{equation*}
Let $\epsilon>0$, $2<r<4$ and write $r=2+2t=2(1-t)+4t$ with $t=t(\epsilon)$ to be chosen later. Then by H\"older's inequality, for any $T_{j-1}<u\leq v \leq T_j$ we have 
\begin{align*}
\EE|M_f(v)-M_f(u)|^r&\leq \big(\EE |M_f(v)-M_f(u)|^2\big)^{1-t}\big(\EE |M_f(v)-M_f(u)|^4\big)^{t}\\
&\leq\left(\sum_{u<n\leq v}\frac{1}{n} \right)^{1-t}\left(\sum_{u<n\leq v}\frac{\tau(n)}{n} \right)^{2t}.
\end{align*}
But by Lemma \ref{uniform divisor lem},
\[
\sum_{u<n\leq v}\frac{\tau(n)}{n}\ll (\log v)^{4/3}(\log(v/u))^{2/3}\ll (\log v)^{4/3}\bigg(\sum_{u<n\leqs v}\frac{1}{n}\bigg)^{2/3}
\]
since $\sum_{u<n\leqs v}1/n\geqs \int_{u+1}^{v+1}dx/x\geqs \tfrac{1}{2}\log (v/u)$ uniformly for $1\leqs u<v$. Therefore, 
\begin{align*}
\EE|M_f(v)-M_f(u)|^r
&\leq 
\bigg(\sum_{u<n\leq v}\frac{1}{n} \bigg)^{1+t/3}(\log T_j)^{8t/3}.
\end{align*}

Thus, Lemma \ref{lemma billingsley} is applicable with $\beta=\frac{r}{4}$, $\alpha=\frac{1+t/3}{2}$ and for some constant $c_0>0$, 
\[u_n=u_n(j)=c_0\frac{(\log T_j)^{\frac{8t/3}{1+t/3}}}{n},\]
 where $T_{j-1}<n\leq T_j$. Here, we remark that the fact $\alpha>1/2$, and hence the use of Lemma \ref{lemma billingsley}, was allowed by the improved error term of $O(x^{1/3})$ in the divisor problem.  In fact, any exponent of $x$ less than $1/2$ would have sufficed here. Now, we have that
\begin{align*}
&\PP\left(\max_{T_{j-1}<T\leq T_j} |M_f(T)-M_f(T_{j-1})|\geq (\log T_{j-1})^{1/2+\epsilon} \right)\\
&\ll \frac{1}{(\log T_{j-1})^{r/2+r\epsilon}}\bigg(\sum_{T_{j-1}<n\leq T_j}\frac{1}{n} \bigg)^{1+t/3}(\log T_j)^{8t/3}
\ll \frac{(\log(T_j/T_{j-1})^{1+t/3}(\log T_j)^{8t/3}}{(\log T_{j-1})^{r/2+r\epsilon}}\\
& \ll
\frac{(j^4-(j-1)^4)^{1+t/3}j^{32t/3}}{j^{2r+4r\epsilon}}
\ll
 \frac{j^{3+t+32t/3}}{j^{4+4t+4r\epsilon}}
 \leqs 
 \frac{1}{j^{1-8t+4r\epsilon}}. 
\end{align*}
Select $8t=4\epsilon$ so that $-8t+4r\epsilon>0$. Hence, by the Borel-Cantelli Lemma the following holds almost surely:
\begin{equation}\label{equacao controle do maximo}
\max_{T_{j-1}<T\leq T_j} |M_f(T)-M_f(T_{j-1})|\ll (\log T_{j-1})^{1/2+\epsilon}.
\end{equation}

Next, we recall that for any $(\log_2 T)^{1/2}\log_3 T\leq V\leq c\log T/\log_2 T$, 
\begin{equation*}
\PP(|M_f(T)|\geq e^V)\leq \exp\left( -(1+o(1))\frac{V^2}{\log (\log T)/V} \right).
\end{equation*}
Select,  $e^V=(\log T_j)^{1/2+\epsilon}$, so that $V=(1/2+\epsilon)\log\log T_j$. Thus
\begin{align*}
\PP(|M_f(T_j)|\geq (\log T_j)^{1/2+\epsilon})&\leq \exp\left( -(1+o(1))(1/2+\epsilon)^2\frac{(\log\log T_j)^2}{\log (\log T_j)/\log\log T_j} \right)\\
&=\exp(-(1+o(1))(1/2+\epsilon)^2 \log\log T_j)\\
&=\exp(-(1+o(1))(1+8\epsilon+4\epsilon^2) \log j).
\end{align*}

Thus, by the Borel-Cantelli Lemma we conclude that almost surely
\begin{equation}\label{equacao controle pontual}
M_f(T_{j-1})\ll (\log T_{j-1})^{1/2+\epsilon}.
\end{equation}
Now, for any arbitrarily large $T$, select $j=j(T)$ such that $T_{j-1}<T\leq T_j$. Then by \eqref{equacao controle do maximo} and \eqref{equacao controle pontual} we conclude that almost surely
\begin{align*}
M_f(T)=M_f(T_{j-1})+(M_f(T)-M_f(T_{j-1}))\ll (\log T_{j-1})^{1/2+\epsilon}\leq (\log T)^{1/2+\epsilon}.
\end{align*}
\end{proof}

\section{Omega bounds}
In this section we prove Theorem \ref{teorema omega bound 1}. We require the following estimate on $y$-smooth numbers less than $x$. 
\begin{lem}\label{lemma u_y} Let $2\leqs y\leqs x$ and let $\Psi(x,y)$ denote the number of integers less than or equal to $x$ all of whose prime factors are less than or equal to $y$. Then for $y\leqs \sqrt{\log x}$ we have 
\begin{equation*}
\Psi(x,y)= \frac{1}{\pi(y)!}\left(\prod_{p\leq y}\frac{1}{\log p}\right)(\log x)^{\pi(y)}\bigg(1+O\Big(\frac{y^2}{\log x\log y}\Big)\bigg).
\end{equation*}
In particular, for fixed $y$ we have $\Psi(x,y)\ll_\epsilon x^\epsilon$ for all $\epsilon>0$. 
\end{lem}
\begin{proof}
This is originally due to Ennola \cite{Enn}. See also Corollary 3.1 of Section III.5.2 of \cite{Ten}.  
\end{proof}
Let $(X_n)_{n\in\mathbb{N}}$ be a sequence of independent random variables. Let $\mathcal{A}$ be an event that is measurable in the sigma algebra $\sigma(X_1,X_2,...)$. We say that $\mathcal{A}$ is a tail event if for any fixed $y\in \mathbb{N}$, $\mathcal{A}$ is independent from $X_1,X_2,...,X_y$. The Kolmogorov zero--one law says that every tail event either has probability $0$ or $1$. Let $\lambda(T)$ be an increasing function such that $\lambda(T)\to\infty$ as $T\to\infty$. In the next Lemma we will show that the event 
\begin{equation*}
\mc{A}_\lambda=\Big\{|M_f(T)| \geq \exp((1+o(1))\lambda(T)),\mbox{ for infinitely many integers }T>0\Big\}
\end{equation*}
is a tail event with respect to the $(f(p))_p$ i.e. that any change on the values $(f(p))_{p\in\mathcal{S}}$ with $\mathcal{S}$  a finite subset of primes, will not change the outcome.

\begin{lem}\label{lemma tail event} Let $\mathcal{A}_{\lambda}$ be as above. Then $\mathcal{A}_{\lambda}$ is a tail event.
\end{lem}
\begin{proof} Let $y>0$ and $f_y(n)$ be the multiplicative function such that for each prime $p$ and any power $m\in\NN$, $f_y(p^m)=f(p^m)\ind_{p>y}$. Let $\mathcal{B}_{y,\lambda}$ be the event in which $|M_{f_y}(T)|\geq \exp((1+o(1))\lambda(T))$ for infinitely many integers $T$. We are going to show that $\mathcal{A}_{\lambda}$ and $\mathcal{B}_{y,\lambda}$ are the same event, and since the values $(f(p))_{p}$ are independent and $\mathcal{B}_{y,\lambda}$ does not depend on the first values $(f(p))_{p\leq y}$, we can conclude that $\mathcal{A}_\lambda$ is a tail event.

Firstly, we will show that $\mathcal{B}_{y,\lambda}\subset \mathcal{A}_\lambda$. Let $g_y(n)$ be the multiplicative function such that at each prime $p$ and any power $m\in\NN$, $g_y(p^m)=\mu(p^m)f(p^m)\ind_{p\leq y}$. Then $f_y=g_y\ast f$ and hence
\begin{align}\label{equacao auxiliar 1}
M_{f_y}(T)=\sum_{n\leq T}\frac{g_y(n)}{\sqrt{n}}M_f(T/n).
\end{align}
Now observe that the set $\{n\in\NN:g_y(n)\neq 0\}$ has at most $2^{\pi(y)}$ elements. Thus, in the event $\mathcal{B}_{y,\lambda}$, by the pigeonhole principle applied to \eqref{equacao auxiliar 1}, we obtain $n\leq \prod_{p\leq y}p$ such that $|M_f(T/n)|\geq \frac{\exp((1+o(1))\lambda(T))}{2^{\pi(y)}}=\exp((1+o(1))\lambda(T))$ for infinitely many integers $T$. Thus $\mathcal{B}_{y,\lambda}\subset \mathcal{A}_{\lambda}$.

Now we will show that $\mathcal{A}_{\lambda}\subset\mathcal{B}_{y,\lambda}$. Firstly we partition: $\mathcal{A}_\lambda=(\mathcal{A}_\lambda\cap \mathcal{C}_y)\cup (\mathcal{A}_\lambda\cap \mathcal{C}_y^c)$, where $\mathcal{C}_y$ is the event in which $M_{f_y}(T)\ll \exp(2\lambda(T))$. Clearly the event $\mathcal{A}_\lambda\cap \mathcal{C}_y^c$ is contained in $\mathcal{B}_{y,\lambda}$. Now let $h_y(n)$ be the multiplicative function such that at each prime $p$ and any power $m\in\NN$, $h_y(p^m)=f(p^m)\ind_{p\leq y }$. Then $f=h_y\ast f_y$ and hence, for any $U>0$
\begin{align*}
M_f(T)=\sum_{n\leq U}\frac{h_y(n)}{\sqrt{n}}M_{f_y}(T/n)+\sum_{U<n\leq T}\frac{h_y(n)}{\sqrt{n}}M_{f_y}(T/n).
\end{align*}
Now we will show that in the event $\mathcal{A}_{\lambda}\cap\mathcal{C}_y$, the second sum in the right hand side above is $o(1)$ for a suitable choice of $U$. Let $u_y(n)$ denote the indicator function of $y$-smooth numbers. Then, by partial integration and Lemma \ref{lemma u_y}:
\begin{align*}
\sum_{U<n\leq T}\frac{h_y(n)}{\sqrt{n}}M_{f_y}(T/n)&\ll \exp(2\lambda(T))\sum_{U<n\leq T}\frac{u_y(n)}{\sqrt{n}}\\
&\ll \exp(2\lambda(T))\left(\frac{(\log U)^{\pi(y)}}{\sqrt{U}}+\frac{1}{2}\int_{U}^\infty\frac{\Psi(x,y)}{x^{3/2}}dx\right)\\
&\ll \exp(2\lambda(T))\int_{U}^\infty \frac{x^\epsilon}{x^{3/2}}dx\\
&\ll \frac{ \exp(2\lambda(T))}{U^{1/2-\epsilon}}.
\end{align*}
Thus, for the choice $U=U(T)=\exp(10\lambda(T))$, we have shown that in the event $\mathcal{A}_\lambda\cap\mathcal{C}_y$
\begin{equation}\label{equacao auxiliar 2}
M_f(T)=\sum_{n\leq U}\frac{h_y(n)}{\sqrt{n}}M_{f_y}(T/n)+o(1).
\end{equation}
Now, by Lemma \ref{lemma u_y}, the set $\{n\leq U:h_y(n)\neq 0\}$ has at most $\Psi(U,y)\ll (\log U)^{\pi(y)}=(10\lambda(T))^{\pi(y)}$ elements. Hence, in the event $\mathcal{A}_{\lambda,L}\cap\mathcal{C}_y$, by the pigeonhole principle applied to \eqref{equacao auxiliar 2}, we always find infinitely many integers $T$ and $n=n(T)\leq U$ such that 
\begin{equation*}
|M_{f_y}(T/n)|\geq \frac{\sqrt{n}\exp((1+o(1))\lambda(T))}{(10\lambda(T))^{\pi(y)}}\geq \exp((1+o(1))\lambda(T)).
\end{equation*}
Thus, $\mathcal{A}_{\lambda}\cap\mathcal{C}_y\subset\mathcal{B}_{y,\lambda}$, and hence $\mathcal{A}_{\lambda}=\mathcal{B}_{y,\lambda}$.
\end{proof}

\begin{proof}[Proof of Theorem \ref{teorema omega bound 1}] We argue as in Lemma 3.2 of \cite{aymone real zeros}. Let $\lambda(T)=L\sqrt{\log\log T}$. We have that the event $\mathcal{A}_\lambda$ is the event $\bigcap_{n=1}^\infty\bigcup_{T=n}^\infty[|M_f(T)|\geq \exp((L+o(1))\sqrt{\log\log T} )]$ and hence 
\begin{align*}
\PP(\mathcal{A}_\lambda)&=\lim_{n\to\infty}\PP\left(\bigcup_{T=n}^\infty[M_f(T)\geq\exp((L+o(1))\sqrt{\log\log T} ) ] \right)\\
&\geq \delta>0,
\end{align*}
where in the second line above we used the Gaussian lower bound \eqref{small v range}. Thus, by the Kolmogorov zero--one law, we conclude that $\PP(\mathcal{A}_\lambda)=1$.
\end{proof}


\begin{thebibliography}{}



\bibitem{ABBRS}L.-P. Arguin, D. Belius, P. Bourgade, M.Radziwill, K. Soundararajan, \emph{Maximum of the Riemann zeta function on a short interval of the critical line} Commun. Pure. Appl. Math {\bf 72} no. 3 (2019), 500--535. 


\bibitem{aymone beyond one half}{M. Aymone, \emph{A note on multiplicative functions resembling the M\"obius function}. Journal of Number Theory, {\bf 212} (2020), 113--121.}

\bibitem{aymone law of iterated logarithm for random dirichlet}{M. Aymone, S. Fr\'ometa, R. Misturini, \emph{Law of the iterated logarithm for random Dirichlet series}, arxiv, 2020}

\bibitem{aymone real zeros}{M. Aymone, \emph{Real zeros of random Dirichlet series}. Electronic comunnications in probability, v. 24, p. 1-8, 2019. }



\bibitem{alexits}{G. Alexits, \emph{Convergence problems of orthogonal series}, Hungarian Acad. Sci., Budapest, 1961. MR 36:1911}



\bibitem{bayart} F. Bayart, \emph{Hardy spaces of Dirichlet series and their composition operators}, Monatsh. Math. {\bf 136} no. 3 (2002), 203--236.


\bibitem{bas}J. Basquin, \emph{Sommes friables de fonctions multiplicatives al\'eatoires}, Acta Arith. {\bf 152} (2012), 243--266.



\bibitem{billingsleyconvergenceofprobabilitymeasures}{\sc P.~Billingsley}, {\em Convergence of probability measures}, Wiley Series in Probability and Statistics: Probability and 
Statistics, John Wiley \& Sons, Inc., New York, second~ed., 1999. \newblock A Wiley-Interscience Publication.   


\bibitem{bohr}H. Bohr, \emph{\"Uber die Bedeutung der Potenzreihen unendlich vieler Variabeln in der Theorie der Dirichletschen reihen $\sum a_n/n^s$}, Nachr. Ges. Wiss. G\"ottingen Math. Phys. Kl. (1913), 441--488.


\bibitem{BHS}  A. Bondarenko, W. Heap, K. Seip \emph{An inequality of Hardy--Littlewood type for Dirichlet polynomials}, J. Number Theory, {\bf 150} (2015), 191--205.



\bibitem{BH}O. F. Brevig, W. Heap, \emph{High pseudomoments of the Riemann zeta function}, J. Number Theory {\bf 197} (2019), 383--410.



\bibitem{CG} B. Conrey and A. Gamburd, \emph{Pseudomoments of the Riemann zeta-function and pseudomagic squares}, J. Number Theory \textbf{117} (2006), 263--278.



\bibitem{Enn}V. Ennola, \emph{On numbers with small prime divisors}, Ann. Acad. Sci. Fenn., Series A. I. {\bf 440}, 16 pp. 



\bibitem{FGH} D. Farmer, S. Gonek, C. Hughes, \emph{ The maximum size of L-functions} J. Reine Angew. Math. (Crelles Journal){\bf  609} (2007) 215--236.


\bibitem{FHK}Y. V. Fyodorov, G. A. Hiary, J. P. Keating. \emph{Freezing Transition, Characteristic Polynomials of Random Matrices, and the Riemann Zeta-Function}, Phys. Rev. Lett., {\bf 108} 170601 (2012).


\bibitem{G}M. Gerspach, \emph{Low pseudomoments of the Riemann zeta function and its powers}. Preprint \href{https://arxiv.org/abs/1909.10224}{arxiv.1909.10224}.



\bibitem{Hal}G. Hal\'asz, \emph{On random multiplicative functions}, Hubert Delange colloquium (Orsay,1982), Publ. Math. Orsay {\bf 83} no. 4 (1983), 74--96.


\bibitem{Ha0}A. J. Harper, \emph{Bounds on the suprema of Gaussian processes, and omega results for the sum of a random multiplicative function}, Ann. Appl. Probab. {\bf 23} no. 2 (2013), 584--616.


 \bibitem{H low}A. Harper, \emph{Moments of random multiplicative functions, I: Low moments, better than squareroot cancellation, and critical multiplicative chaos}, Forum of Math, Pi, {\bf 8} (2020).


\bibitem{H high}A. J. Harper, \emph{Moments of random multiplicative functions, II: High moments}, Algebra and Number Theory,  {\bf 13} no. 10 (2019), 2277--2321.

 
 \bibitem{Ha4} A. J. Harper, \emph{On the partition function of the Riemann zeta function, and the Fyodorov--Hiary--Keating conjecture}, preprint, \href{https://arxiv.org/abs/1906.05783}{arxiv.1906.05783}.
 


\bibitem{H} W. Heap, \emph{Upper bounds for $L^q$ norms of Dirichlet polynomials with small $q$}, J. Functional Analysis, {\bf 275} No. $\!\!\!$ 9 (2018) 2473--2496.



\bibitem{J}M. Jutila, \emph{On the value distribution of the zeta-function on the critical line}, Bull. London Math. Soc. {\bf 15} (1983), 513--518.

  
\bibitem{Kol}A. Kolmogorov, \emph{Uber das Gesetz des iterierten Logarithmus}. Math. Ann. {\bf 101} (1929), 126--135.  
   

\bibitem{LTW} Y. Lau, G. Tenenbaum, J. Wu, \emph{Mean values of random multiplicative functions}, Proc. Am. Math. Soc. {\bf 141} no. 2 (2013), 409--420.


\bibitem{MV}H. L. Montgomery, R. C. Vaughan, \emph{Multiplicative Number Theory I: Classical Theory}, 1st ed., Cambridge University Press, Cambridge, 2007.  


\bibitem{N}J. Najnudel, \emph{ On the extreme values of the Riemann zeta function on random intervals of the critical line}, Probab. Theory Relat. Fields {\bf 172}, no. 1Ð2 (2018) 387--452.




\bibitem{R}M. Radziwi\l\l,\, \emph{Large deviations in Selberg's central limit theorem}. Preprint \href{https://arxiv.org/abs/1108.5092}{arxiv.1108.5092}.

\bibitem{shiryaev}{\sc A.~N. Shiryaev}, {\em Probability}, vol.~95 of Graduate Texts in Mathematics, Springer-Verlag, New York, second~ed., 1996.
\newblock Translated from the first (1980) Russian edition by R. P. Boas.



\bibitem{S}K. Soundararajan, \emph{Moments of the Riemann zeta function}, Annals of Math. {\bf 170} no. 2 (2009), 981--993.


\bibitem{Ten}G. Tenenbaum, \emph{Introduction to analytic and probabilistic number theory}, {Cambridge studies adv. math.} 46, 1995.



\bibitem{T} E. C. Titchmarsh, \emph{The Theory of the Riemann Zeta-Function}, Second Edition, The Clarendon Press, Oxford University Press, 1986.


\bibitem{W} A. Wintner, \emph{Random factorizations and Riemann's hypothesis}, Duke Math. J. {\bf 11} no. 2 (1944), 267--275.

  
  
\bibitem{Weiss} F. B. Weissler, \emph{Logarithmic Sobolev inequalities and hypercontractive estimates on the circle}, J. Funct. Anal. {\bf 37} no. 2 (1980), 218--234.


  
\end{thebibliography}
\end{document}